\documentclass[11pt, twoside, letterpaper]{article}
\usepackage{tikz}
\usepackage{verbatim}
\usepackage{color}
\definecolor{MyDarkGreen}{rgb}{0,0.45,0}
\definecolor{MyDarkRed}{rgb}{0.9,0,0}

\usepackage{amsmath}
\usepackage{amssymb}
\usepackage{amsthm}
\usepackage{bbm}

\newtheorem{theorem}{Theorem}[section]
\newtheorem{lemma}[theorem]{Lemma}

\newtheorem{corollary}[theorem]{Corollary}

\usepackage[margin=1in]{geometry}

\title{Monotone convergence of spreading processes on networks}
\author{Gadi Fibich, Amit Golan, and Steven Schochet}

\date{\today}

\makeatletter
\let\inserttitle\@title
\makeatother
\usepackage{fancyhdr}
\begin{document}

	\maketitle
	\begin{abstract}
		We analyze the Bass and SI~models for the spreading of innovations and epidemics, respectively, on homogeneous complete networks, circular networks, and heterogeneous complete networks with two homogeneous groups. We allow the network parameters to be time dependent, which is a prerequisite for the analysis of optimal strategies on networks.  Using a novel top-down analysis of the master equations, we present a simple proof for the monotone convergence of these models to their respective infinite-population limits. This leads to explicit expressions for the expected adoption or infection level in the Bass and SI~models, respectively, on infinite homogeneous complete and circular networks, and on heterogeneous complete networks with two homogeneous groups with time-dependent parameters. 
			\end{abstract}

	\section{Introduction}
	
	
	Spreading processes on networks have attracted the attention of researchers  
	in physics, mathematics, biology, computer science, social sciences, economics, and management science,
	as it concerns the spreading of ``items'' ranging from diseases and computer viruses to rumors, information, opinions, technologies, and 
	innovations~\cite{Albert-00,Anderson-92,Jackson-08,Pastor-Satorras-01,Strang-98}. In this study, we focus on two prominent network models: The Bass model for the adoption of innovations~\cite{Bass-69}, and the Susceptible-Infected (SI) model for the spread of epidemics~\cite{Epidemics-on-Network-17}. 
   These models were originally formulated as compartmental models, in which the 
   population is divided into two compartments: adopters/infected and nonadopters/susceptibles. In recent years, research has shifted to studying these models on networks.

	In the Bass and SI models on networks, the adoption/infection event by each node is stochastic. Since a direct analysis of stochastic particle models is hard, 	  
	there has been a considerable research effort to derive a deterministic ODE 
	for the macroscopic behavior of the expected adoption/infection level as a function of time.
	Niu~\cite{Niu-02} {\em derived} the ODE for the infinite-population limit of the Bass model on homogeneous complete networks. The approach in~\cite{Niu-02}, however, does not extend to other types of networks. Fibich and Gibori obtained an explicit expression for the expected adoption level in the Bass model on infinite circles~\cite{OR-10}. They did not prove rigorously, however, that this expression is the limit of the Bass model on circles with $M$~nodes as $M\to\infty$. In~\cite{DCDS-23}, Fibich et al.\ rigorously derived the 
	infinite-population limit of the Bass model on homogeneous complete networks,
	on heterogeneous complete networks with $K$~groups, and on circular networks, and also the rate of convergence for these three cases. We are not aware of similar convergence results for the SI model on networks.

	   The common theme of the above studies has been to predict the expected adoption/infection level as a function of time, and to analyze how it is affected by the network structure and parameters. We note, however, 
	   that another important application of the (compartmental) Bass and SI models has been to compute optimal strategies that influence the spreading process
	   im a desired fashion. For example, one can use the Bass model to compute optimal promotional campaigns that maximize the profits~\cite{optimal-promotion}. Similarly, the SI~model can be used to compute optimal government restrictions that minimize disease spread while keeping the economy healthy~\cite{Balderrama2022-qa}.  So far,
	   these optimal-control problems have only been studied in the context of compartmental Bass and SI models, which implicitly assume that the social network is a complete homogeneous network.
 
 In order to apply the machinery of optimal-control theory to spreading on networks, one first needs to derive a single ODE, or a small system of ODEs, for the macroscopic dynamics. Moreover, if one want to allow for time-dependent optimal strategies, the network parameters should be allowed to be time-dependent. 
	In this paper, we present the first rigorous derivation of a single ODE
	for the expected adoption/infection level in the Bass and SI~models on 
	networks with time-varying parameters. These ODEs, in turn, are 
	used in a companion study~\cite{optimal-promotion} to compute and analyze optimal promotional strategies in the Bass model on networks.

To the best of our knowledge, this paper presents 
 first derivation of infinite-population limit of the Bass model on networks with time-varying parameters, and also the first derivation of infinite-population limit of the SI~model on networks, with and without time-varying parameters.
From a methodological point of view, this paper presents the first unified treatment of the Bass and SI~models on networks. In addition, it introduces a novel ``top-down'' analysis of the (bottom-up) master equations,
where one proves the desired property (e.g., monotonicity) at the ``top'' level of the master equation for the whole population, and then proves that the property remains valid as the number of nodes is reduced one at a time, until reaching the desired ``bottom''  level of the master equation for a single node.  Finally, proving the convergence of the model as the population becomes infinite using the monotonicity property is much simpler than the approach used in~\cite{DCDS-23}. Since this monotonicity-based proof also applies to networks with  time-varying parameters, it is potentially 
applicable to a wider variety of networks. In particular, this is a prerequisite  for optimal control applications on networks.

	The paper is organized as follows: In Section~\ref{sec:bass/si} we introduce the unified Bass/SI~model on networks.  In section~\ref{sec:complete}, we analyze the Bass/SI~model on complete homogeneous networks with time-varying parameters. We first show that the expected adoption/infection level is monotonically increases in~$M$.
	We then show that the expected adoption/infection level converges monotonically as $M \to \infty$, and 
	compute this limit  explicitly. In sections~\ref{sec:1D_1_sided} and~\ref{sec:groups}, we obtain similar results for the Bass/SI~model on homogeneous circular networks and on heterogeneous complete networks with two homogeneous groups, respectively.
	
	Finally, we note that the methodology and results of this study have the potential to be extended to other types of networks (Cartesian, random, \dots ), to hypernetworks, and to other types of spreading models (SIS, SIR, Bass-SIR, etc.)~\cite{Epidemics-on-Network-17,Bass-SIR-model-16, Bass-SIR-SIAP}.

	\section{Bass/SI~model on networks}
	\label{sec:bass/si}
	
	The Bass model describes the adoption of new products or innovations within a population. In this framework, all individuals start as non-adopters and can transition to becoming adopters due to two types of influences: external factors, such as exposure to mass media, and internal factors where individuals are influenced by their peers who have already adopted the product.
	The SI~model is used to study the spreading of infectious diseases within a population. In this model, some individuals are initially infected
        (the ``patient zero'' cases), all subsequent infections occur through internal influences, whereby infected individuals transmit the disease to their susceptible peers, and infected individuals remain contagious indefinitely.
In both models, once an individual becomes an adopter/infected, it remains so at all later times.  In particular, 
she or he remain ``contagious'' forever.  
	The difference between the SI~model and the Bass model is the lack of external influences in the former, and the lack of initial adopters in the latter. 
	
	It is convenient to unify these two models into a single model, 
	the Bass/SI~model on networks, as follows. 
	Consider $M$~individuals, denoted by ${\cal M}:=\{1, \dots, M\}$. 
	We denote by $X_j(t)$ the state of individual~$j$ at time~$t$, so that 
	\begin{equation*}
		X_j(t)=\begin{cases}
			1, \qquad {\rm if}\ j\ {\rm is \ adopter/infected \ at\ time}\ t,\\
			0, \qquad {\rm otherwise,}
		\end{cases}
		\qquad j \in \cal M.
	\end{equation*} 
	The initial conditions at $t=0$ are stochastic, so that 
	\begin{subequations}
		\label{eqs:Bass-SI-models-ME}
		\begin{equation}
			\label{eq:general_initial}
			X_j(0)=	X_j^0 \in \{0,1\}, \qquad j\in {\cal M},
		\end{equation}
		where
		\begin{equation}
			\mathbb{P}(X_j^0=1) =I_j^0, \quad 
			\mathbb{P}(X_j^0=0) =1-I_j^0,\quad I_j^0 \in [0, 1],  \qquad 
			j \in \cal M,
		\end{equation}
		and 
		\begin{equation}
			\label{eq:p:initial_cond_uncor-two_sided_line}
			\mbox{the random variables $\{X_j^0 \}_{j \in \cal M}$ are independent}.
		\end{equation} 
		{\em Deterministic initial conditions} are a special case where
		$I_j^0 \in \{0,1\}$. 
		
		So long that $j$ is a nonadopter/susceptible, its adoption/infection rate at time~$t$ is
		\begin{equation}
			\label{eq:lambda_j(t)-Bass-model-heterogeneous-tools}
			\lambda_j(t) = p_j(t)+\sum\limits_{k\in {\cal M}} q_{k,j}(t) X_{k}(t),
			\qquad j \in {\cal M}.
		\end{equation}
		Here, $p_j(t)$ is the rate of external influences on~$j$, and~$q_{k,j}(t)$ is the rate of 
	internal influences by~$k$ on~$j$ at time $t$, provided that $k$ is already an adopter/infected. Once~$j$ becomes an adopter/infected, it remains so at all later times.\,\footnote{i.e., the only admissible transition is 
		$X_j=0 \to X_j=1$.}
	Hence, as $ \Delta t \to 0$,
		\begin{equation}
			\label{eq:general_model}
			\mathbb{P} (X_j(t+\Delta  t )=1  \mid   {\bf X}(t))=
			\begin{cases}
				\lambda_j(t) \, \Delta t , &  {\rm if}\ X_j(t)=0,
				\\
				1,\hfill & {\rm if}\ X_j(t)=1,
			\end{cases}
			\qquad 	j \in {\cal M},
		\end{equation}
		where ${\bf X}(t) := \{X_j(t)\}_{j \in \cal M}$ 
		is the state of the network at time~$t$,
		and
		\begin{equation}
			\label{eq:Bass-SI-models-ME-independent}
			\mbox{the random variables $\{X_j(t+\Delta  t )  \mid   {\bf X}(t) \}_{j \in \cal M}$ are independent}.
		\end{equation}  
		%
		%
	\end{subequations}
In the Bass model there are no adopters when the product is first introduced into the market, and so $I_j^0 \equiv 0$.
In the SI~model there are only internal influences for $t>0$, and so $p_j(t) \equiv 0$.

	The quantity of most interest is the expected  adoption (infection) level
	\begin{equation}
		\label{eq:number_to_fraction-general}
		f(t):= 
		\frac{1}{M} \sum_{j=1}^{M} f_j(t),
	\end{equation}
	where $f_j :=\mathbb{E}[X_j]$ in the adoption/infection probability of node~$j$.

	\subsection{Master equations}
	
	The key tool in the analysis of the  Bass/{\rm SI} model~\eqref{eqs:Bass-SI-models-ME}
	are the master equations. 
	Let $ \emptyset \not= \Omega \subset  {\cal M}$ be a nontrivial subset of the nodes,  let
	$
	\Omega^{\rm c}:={\cal M} \setminus \Omega,
	$  and let 
	\begin{equation}
		\label{eq:def-S_Omega-ME}
		S_{\Omega}(t):=\{X_{m}(t)=0,~ \forall m \in \Omega \},
		\qquad  [S_{\Omega}](t)
		:= \mathbb{P}(S_{\Omega}(t)),
	\end{equation}
	denote the event that all nodes in~$\Omega$ are nonadopters at time~$t$, and
	the probability of this event, respectively.
	To simplify the notations, we introduce the notation
	\begin{equation*}
		S_{\Omega_1,\Omega_2}:=S_{\Omega_1 \cup \Omega_2}, \qquad
		\Omega_1,\Omega_2 \subset \cal M.
	\end{equation*}
	Thus, for example, $	S_{\Omega,k}:=S_{\Omega \cup \{k\}}$.
	We also denote the sum of the external influences on the nodes in~$\Omega$
	and the sum of the internal influences by node $k$ on the nodes in~$\Omega$
	by 
	$$
	p_\Omega(t):=\sum_{m  \in \Omega} p_{m}(t),
	\qquad 
	q_{k,\Omega}(t):=\sum_{m \in \Omega} q_{k,m}(t),
	$$
	respectively. 
	
	\begin{theorem}[\cite{MOR-22}]
		\label{thm:master-eqs-general}
		The master equations for the Bass/{\rm SI} model~\eqref{eqs:Bass-SI-models-ME} are
		\begin{subequations}
			\label{eqs:master-eqs-general}
			\begin{equation}
				\label{eq:master-eqs-general}
				\frac{d [S_{\Omega}]}{dt}  = 
				-\bigg( p_\Omega(t)
				+\sum_{k \in \Omega^{\rm c}} q_{k,\Omega}(t)\bigg)[S_{\Omega}] 
				+\sum_{k \in \Omega^{\rm c}}q_{k,\Omega}(t)\,  [S_{\Omega,k}],
			\end{equation}
			subject to the initial conditions
			\begin{equation}
				\label{eq:master-eqs-genera-icl}
				[S_{\Omega}](0)=[S_{\Omega}^0], \qquad [S_{\Omega}^0]:= \prod_{m \in \Omega} (1-I_m^0), 
			\end{equation}
			for all $\emptyset\not=\Omega \subset {\cal M}$.
		\end{subequations}
	\end{theorem}
	If we can solve these  $2^M-1$ equations, then we have~$f(t)$ from~\eqref{eq:number_to_fraction-general} and~$f_j = 1-[S_j]$.

	\begin{corollary}
		\label{cor:[S_Omega]<=[S_Omega^0]e^-pt}
		Consider the Bass/{\rm SI} model~\eqref{eqs:Bass-SI-models-ME}.
		Let $\emptyset \not=\Omega \subset \cal M$. 
		Then 
		\begin{equation}
			\label{eq:[S]<e^-pt}
			[S_{\Omega}](t) \le [S_{\Omega}^0] e^{-\int_0^t p_\Omega}, \qquad t \ge 0.
		\end{equation}
	\end{corollary}
	\begin{proof}
		Since $q_{k,\Omega} \ge 0$ and $[S_{\Omega}] - [S_{\Omega,k}] \ge 0$,
		the result follows from equation~\eqref{eqs:master-eqs-general}
		for~$[S_{\Omega}]$, together with the fact that $\frac{dx}{dt}-p(t)x\le0$, $x(0)=x_0$ implies that $x(t)\le y(t)$, where $y(t)$ is
                  the solution of $\frac{dy}{dt}-p(t)y=0$, $y(0)=0$, which can be proven similarly to Lemma~\ref{lem:odeineq} below.
	\end{proof}

	\begin{corollary}
		\label{cor:M-1}
		Consider the Bass/{\rm SI} model~\eqref{eqs:Bass-SI-models-ME}.
		Let $M=1$. 
		Then 
		\begin{equation}
			\label{eq:[S]M=1}
			[S](t;p(t),I^0,M=1)   = (1-I^0) e^{-\int_0^t p}, \qquad t \ge 0.
		\end{equation}
	\end{corollary}
	\begin{proof}
		This follows from eq.~\eqref{eqs:master-eqs-general}.
	\end{proof}

	\section{Bass and SI~models on complete networks}
	\label{sec:complete}
	
	Consider a complete homogeneous network where everyone is connected to each other, all the nodes have the same initial condition, and all the nodes and all the edges have the same weights.
	Thus,\,\footnote{The internal influences have been normalized, so that the maximal  internal influence $\sum_{k \in \cal M}q_{k,j}(t)$ is independent of the network size~$M$.}  
	\begin{subequations}
		\label{eqs:complete-network}
		\begin{equation}
			\label{eq:p_j_q_j_complete-homog}
			I_j^0\equiv I^0,\qquad p_j(t)\equiv p(t),\qquad
			q_{k,j}(t)\equiv  \frac{q(t)}{M-1} \mathbbm{1}_{ k \not= j},
			\qquad  \qquad k,j \in {\cal M}.
                      \end{equation}     
		Hence, the adoption rate of each of the nonadopting nodes is, see~\eqref{eq:lambda_j(t)-Bass-model-heterogeneous-tools},  
		\begin{equation}
			\label{eq:Bass-model-homog-complete-lambda_j}
			\lambda^{\rm complete}(t) := p(t)+ \frac{q(t)}{M-1} N(t), 
			\qquad N:=\sum_{k \in \cal M} X_k,
		\end{equation}
		where $N(t)$
		is the number of adopters/infected in the network.
		Note that we allow the weights to be time-dependent, which is essential for the 
		analysis of time-dependent promotional strategies on networks~\cite{optimal-promotion}. We  assume that the parameters satisfy 
		\begin{equation}
			0 \le I^0<1, \qquad 	q(t)>0, \quad  p(t) \ge 0, \quad t>0,
		\end{equation}
		and 
		\begin{equation}
			I^0>0  \qquad  \text{or} \qquad  p(t)>0, \quad t>0. 
		\end{equation}
	\end{subequations}
Furthermore, we assume that $p(t)$ and $q(t)$ are piecewise continuous.
	
	We denote the expected adoption/infection level in the Bass/SI~model 
	on the complete network~\eqref{eqs:complete-network} by~$f^{\rm complete}$.
	Specifically, in the case of the Bass model,
	\begin{equation}
		I^0=0, \qquad 	q(t)>0, \quad  p(t) > 0, \quad t>0,
	\end{equation}
	and the expected adoption level 
	is $f^{\rm complete}_{\rm Bass}:=f^{\rm complete}(\cdot,I^0=0)$.
	In the case of the SI~model,
	\begin{equation}
		0<I^0<1, \qquad 	q(t)>0, \quad  p(t) \equiv 0, \quad t>0,
	\end{equation}
	and expected infection level 
	is $f^{\rm complete}_{\rm SI}:=f^{\rm complete}(\cdot,p=0)$.
	Note that on the complete network~\eqref{eqs:complete-network}, 
	$f^{\rm complete}_{\rm Bass}$ and $f^{\rm complete}_{\rm SI}$ are directly related, as
		\begin{equation*}
			f^{\rm complete}_{\rm SI}(t;q(t),I^0,M) = I^0+
			\left(1-I^0 \right) 
			f^{\rm complete}_{\rm Bass}\left(t;\widetilde{p}(t) ,\widetilde{q}(t),\widetilde{M}\right),
		\end{equation*}
		where $\widetilde{p}(t):=\frac{MI^0}{M-1} q(t)$,
		$\widetilde{q}(t) :=  \frac{\widetilde{M}-1}{M-1}q(t)$, and~$\widetilde{M}:=M(1-I^0)$
               which follows from the fact that the internal influence of the initial infected individuals on the remaining individuals can
                be viewed equivalently as an external influence on the smaller network from which the initial infected individuals are excluded.
	
	\subsection{Monotone convergence of~$f^{\rm complete}$}
			
	\label{sec:complete-infinite}
		Consider the  Bass/SI~model~(\ref{eqs:Bass-SI-models-ME},\ref{eqs:complete-network}) 
	on a  complete network. As the network size~$M$ increases, each nonadopter can be influenced by more and more adopters, but the influence rate~$q_{k,j} = \frac{q(t)}{M-1}$ of each adopter decays. 
	Therefore, a priori, it is not clear whether $f^{\rm complete}$  should be
	monotonically decreasing or increasing in~$M$. 
	The following lemma settles this issue:
	\begin{lemma}
		\label{lem:f_complete_monotone_in_M}
		Let $t>0$. 
		Then $f^{\rm complete}(t;p(t),q(t),I^0,M)$  is  monotonically increasing in~$M$.
	\end{lemma}
\begin{proof}
	See Section~\ref{sec:f^complete-monotone-in-M}.
\end{proof}

	Using the monotonicity in~$M$, we can prove the convergence 
	of $f^{\rm complete}$  as $M \to \infty$
and compute its limit:  
\begin{theorem}
	\label{thm:Niu(t)}
	Consider the Bass/{\rm SI}  model~{\rm (\ref{eqs:Bass-SI-models-ME},\ref{eqs:complete-network})} on a complete network.
	Then 
	\begin{equation}
		\label{eq:f_Bass-Niu(t)}
		\lim_{M \to \infty} f^{\rm complete}(t;p(t),q(t),I^0,M) = f^{\rm compart.}(t;p(t),q(t),I^0),
	\end{equation}
	where $f^{\rm compart.}$ is the solution of the equation
	\begin{equation}
		\label{eq:ODE-for-f^complete_infty}
		\frac{df}{dt} = (1-f)\big(p(t)+q(t)f\big), \qquad f(0) = I^0.
	\end{equation}
\end{theorem}
\begin{proof}
	See Section~\ref{sec:proof-thm-Niu(t)}.
\end{proof}
From Lemma~\ref{lem:f_complete_monotone_in_M} and Theorem~\ref{thm:Niu(t)}
we have 
	\begin{corollary}
	\label{cor:monotone-convergence-f_complete_to_Bass}
	The convergence of~$f^{\rm complete}$ to~$f^{\rm compart.}$ is monotone in~$M$. 
\end{corollary}

We can also use the monotonicity to obtain lower and upper bounds for~$f^{\rm complete}$:
\begin{corollary}
	\label{cor:f_complete<f_Bass}
	Consider the Bass/SI  model~{\rm (\ref{eqs:Bass-SI-models-ME},\ref{eqs:complete-network})} on a complete network. Then 
	\begin{equation}
		\label{eq:f_complete<f_Bass}
		1-	(1-I^0)e^{-\int_0^t p}< f^{\rm complete}(t;p(t),q(t),I^0,M)<f^{\rm compart}(t;p(t),q(t),I^0), \qquad 
		t>0, \quad M=2,3, \dots
	\end{equation}
\end{corollary}
\begin{proof}
	This follows from Lemma~\ref{lem:f_complete_monotone_in_M}, eq.~\eqref{eq:[S]M=1}, 
	and Theorem~\ref{thm:Niu(t)}.
\end{proof}

\subsection{Time-independent parameters}

When $p$ and $q$ are independent of time and $I^0=0$, we obtain 
from Theorem~\ref{thm:Niu(t)} the well-known compartmental limit
\begin{equation}
	\label{eq:f^complete_Bass_-lim-time-independent}
	\lim_{M \to \infty} f^{\rm complete}_{\rm Bass}(t;p,q,M) =
	f^{\rm compart.}_{\rm Bass}(t;p,q), \qquad 
	f^{\rm compart.}_{\rm Bass}  :=\frac{1-e^{-(p+q)t}}{1+\frac{q}{p}e^{-(p+q)t}}, 
\end{equation}
where $f^{\rm compart.}_{\rm Bass}$ is the solution of the compartmental Bass model~\cite{Bass-69}
$$
f'(t)=(1-f)(p+qf), \qquad f(0)=0.
$$
The monotone convergence of~$f^{\rm complete}_{\rm Bass}$ to~$f^{\rm compart.}_{\rm Bass}$ is illustrated in Figure~\ref{fig:f_complete_converge2_f_bass}.

\begin{figure}[!h]
	\begin{center}
		\scalebox{0.6}{\includegraphics{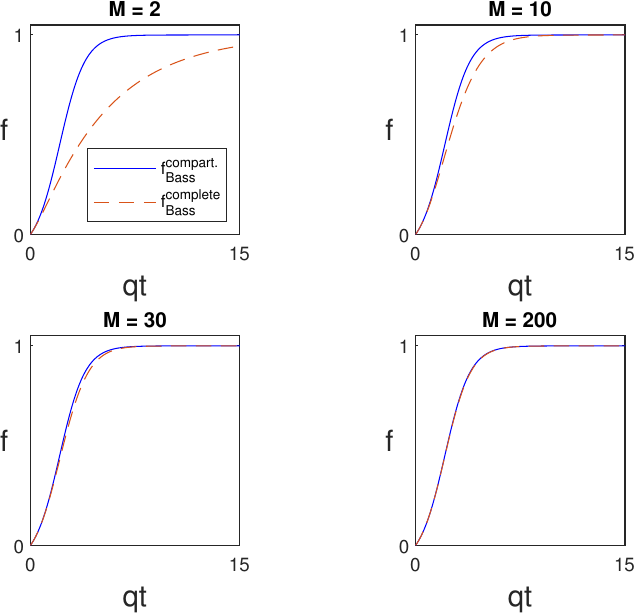}}
		\caption{Monotone convergence of~$f^{\rm complete}_{\rm Bass}$ (dashes) to~$f^{\rm compart.}_{Bass}$ (solid).  Here~$\frac{q}{p} = 10$, $I^0=  0$, and $M=2,10,30,200$.}
		\label{fig:f_complete_converge2_f_bass}
	\end{center}
\end{figure}
Similarly, when $p=0$ and $q$ is independent of time, we obtain
\begin{equation}
	\label{eq:f^complete_SI-lim-time-independent}
	\lim_{M \to \infty} f^{\rm complete}_{\rm SI}(t;q,I^0,M) =
	f^{\rm compart.}_{\rm SI}(t;q, I^0), \qquad 
	f^{\rm compart.}_{\rm SI}  := \frac{1}{1+(\frac1{I^0}-1)e^{-qt}},
\end{equation}
where $f^{\rm compart.}_{\rm SI}$ is the solution of the compartmental SI~model
$$
f'(t)=q(1-f)f, \qquad f(0)=I^0.
$$

The limit~\eqref{eq:f^complete_Bass_-lim-time-independent} was proved in~\cite{DCDS-23,Niu-02}.
To the best of our knowledge, this is the first rigorous derivation of the limit~\eqref{eq:f^complete_SI-lim-time-independent}. 
Furthermore, this is the first proof that $f^{\rm complete}_{\rm Bass}$ and $f^{\rm complete}_{\rm SI}$ converge {\em monotonically} to their respective limits.

\subsection{Proof of Lemma~\ref{lem:f_complete_monotone_in_M} and Theorem~\ref{thm:Niu(t)}} 

	\subsubsection{Reduced master equations}
	
	As noted, there are $2^M-1$ master equations for $\{[S_{\Omega}] \}_{ \Omega \subset {\cal M}}$, see Theorem~\ref{thm:master-eqs-general}.
	Because of the symmetry of the complete network~\eqref{eqs:complete-network}, however, 
	$[S_{\Omega}]$  only depends on the number of nodes in~$\Omega$, and not
	on the identity of the nodes in~$\Omega$. 
	Therefore, we can denote by
	\begin{equation}
		\label{eq:[S^n]-complete}
		[S^n]:= [S_{\Omega}  \large\mid  |\Omega|=n]
	\end{equation}
	the probability that for any given subset of $n$~nodes,  all its nodes are nonadopters at time~$t$. This substitution replaces the $2^M-1$ master equations~\eqref{eqs:master-eqs-general} for $\{[S_{\Omega}] \}_{ \Omega \subset {\cal M}}$ with a reduced system of $M$~equations for~$\{[S^n]\}_{n=1}^M$:
	\begin{lemma}
		\label{lem:master-eq-homog-complete}
		The reduced master equations for the Bass/{\rm SI}  model~{\rm (\ref{eqs:Bass-SI-models-ME},\ref{eqs:complete-network})} on a complete network are 
		\begin{subequations}
			\label{eqs:master-homog-complete}
			\begin{align}
				\label{eq:master-homog-complete}
				\frac{d[S^n]}{dt}& = -n\left(p(t)+q(t)\frac{M-n}{M-1}\right)[S^n]+nq(t)\frac{M-n}{M-1}[S^{n+1}], \qquad n=1,\dots,M-1,
				\\	
				\label{eq:master-homog-complete-M}
				\frac{d[S^M]}{dt}& =-Mp(t) [S^M],
			\end{align}
			subject to the initial conditions
			\begin{equation}
				\label{eq:master-homog-complete-IC}
				[S^n](0)=(1-I^0)^n,\qquad n\in {\cal M}.
			\end{equation}
		\end{subequations}
	\end{lemma}
	\begin{proof}
		This follows from the substitution of~\eqref{eqs:complete-network} and~\eqref{eq:[S^n]-complete} in the master equations~\eqref{eqs:master-eqs-general}.
	\end{proof}
	
	\subsubsection{Monotonicity of $\{[S^n]\}$ in $M$}
	\label{sec:f^complete-monotone-in-M}
	
	Let us recall the following auxiliary result:
	\begin{lemma}\label{lem:odeineq}
		\label{lem:dy_dt+cy>0}
		Let $\alpha(t):\mathbb{R}\to\mathbb{R}$ be piecewise continuous, and let  $y(t)$ satisfy the differential inequality   
		$$
		\frac{dy}{dt}+ \alpha(t)y >0, \quad t>0, \qquad y(0) = 0.
		$$
		Then $y(t)>0$ for $t>0$.
	\end{lemma}  
	\begin{proof}
		Multiplying the differential inequality by the integrating factor $e^{\int_0^t \alpha(s)\, ds}$ gives
		$$
		\frac{d}{dt} \left(e^{\int_0^t \alpha(s)\, ds}y\right) >0.
		$$
		Integrating between zero and $t$ and using the initial condition gives the result.
	\end{proof}

We are now ready to prove Lemma~\ref{lem:f_complete_monotone_in_M}:
	\begin{proof}[Proof of Lemma~\ref{lem:f_complete_monotone_in_M}]
		Let
		\begin{equation}
			\label{eq:y_n=S^n](t;M)-[S^n](t;M+1)}
			y_n(t):= [S^n](t;M)-[S^n](t;M+1),
			\qquad n=1, \dots, M,
		\end{equation}
		where $[S^n](t;M)$ is the solution of the master equations~\eqref{eqs:master-homog-complete}. Then
		\begin{subequations}
			\label{eqs:d_dt_y_M-proof-monotonicity}
			\begin{equation}
				\frac{dy_M}{dt} +Mp y_M =  z_M(t), \qquad 
				y_M(0) = 0,
			\end{equation}
			where 
			\begin{equation}
				z_M :=  q\Big([S^M](t;M+1)-[S^{M+1}](t;M+1)\Big).
			\end{equation}
		\end{subequations}
 Note that 
		$
		[S^M](t;M+1) -[S^{M+1}](t;M+1)= [IS^{M}](t;M+1)>0$,
              where $[IS^M](t,M+1)$ denotes the probability that exactly one of $M+1$ nodes is an adopter/infected.
		Hence $z_M(t)>0$ for $t>0$. Therefore, applying Lemma~\ref{lem:dy_dt+cy>0} to~\eqref{eqs:d_dt_y_M-proof-monotonicity} shows that
		\begin{equation}
			\label{eq:y_M(t)>0-complete}
			y_M(t)>0, \qquad t>0.
		\end{equation}
		
		We can rewrite the master equations~\eqref{eqs:master-homog-complete} for $ n=1,\dots,M-1$ as
		$$
		\frac{d [S^n]}{dt}(t;M) = -n\left(p+q_M^n\right)[S^n](t;M)+nq_M^n[S^{n+1}](t;M), \quad [S^n](0;M)=1, \quad  q_M^n:=q\frac{M-n}{M-1}.
		$$
		Similarly,
		$$
		\frac{d [S^n]}{dt}(t;M+1) = -n\left(p+q_{M+1}^n\right)[S^n](t;M+1)+nq_{M+1}^n[S^{n+1}](t;M+1), \quad [S^n](0;M+1)=1.
		$$
		Taking the difference of these two equations gives
		\begin{subequations}
			\label{eqs:d_dt_y_n-proof-monotonicity}
			\begin{equation}
				\frac{d y_n}{dt} +n\left(p+q_{M+1}^n\right) y_n = nq_{M+1}^n y_{n+1}+ z_n(t), \quad 
				y_n(0) = 0, \qquad n = 1, \dots, M-1,
			\end{equation}
			where
			\begin{equation}
				z_n = n (q_M^n-q_{M+1}^n) \big( -[S^n](t;M)+[S^{n+1}](t;M) \big).
			\end{equation}
		\end{subequations}
		Since 
		$$
		q_M^n-q_{M+1}^n= q\frac{M-n}{M-1}-q\frac{M+1-n}{M} = q\frac{1-n}{(M-1)M} \le 0,
		$$
		and
		$$
		-[S^n]+[S^{n+1}]=  -[IS^{n}]<0, \qquad t>0,
		$$
		we have that $ z_n\ge 0$ for $t>0$. 
		Therefore, applying Lemma~\ref{lem:dy_dt+cy>0} to~\eqref{eqs:d_dt_y_n-proof-monotonicity} shows that
		\begin{equation}
			\label{eq:y_n+1(t)>0-complete}
			y_{n+1}(t)>0, \quad t>0  \qquad \Rightarrow \qquad y_{n}(t)>0, \quad t>0, \qquad\qquad n = 1, \dots, M-1.
		\end{equation}
		
		From relations~\eqref{eq:y_M(t)>0-complete} and~\eqref{eq:y_n+1(t)>0-complete} 
		we get by reverse induction on~$n$ that
		\begin{equation}
			\label{eq:[S^n]-monotone-in-M}
			y_n(t)>0, \qquad n \in {\cal M},
		\end{equation}
		i.e., that $\{[S^n](t;M)\}_{n \in \cal M}$ are monotonically decreasing in~$M$.
		In particular,
		$$
		0<y_1(t) = [S](t;M)- [S](t;M+1) = f^{\rm complete}(t;M+1)-f^{\rm complete}(t;M).
		$$
	\end{proof}
	
		\subsubsection{Convergence of~$f^{\rm complete}$}
	   \label{sec:proof-thm-Niu(t)}

	Next, we utilize the monotonicity of~$\{[S^n]\}$ in~$M$ to prove the 
	convergence of $f^{\rm complete}$:
	\begin{proof}[Proof of Theorem~\ref{thm:Niu(t)}]
		Consider the master equations~\eqref{eqs:master-homog-complete}.
		If we formally fix~$n$ and let $M\to\infty$, we get the ODE  
		\begin{equation}
			\label{eq:master_inf-hom}
			\frac{d [S^n_{\infty}]}{dt}=-n(p+q)[S^n_{\infty}]+nq[S^{n+1}_{\infty}], \qquad [S^n_{\infty}](0) = (1-I_0)^n, \qquad n \in \mathbb{N}.
		\end{equation}
		This does not immediately imply that $\lim_{M \to \infty} [S^n] =  [S^n_{\infty}]$.
		Indeed, this limit does not follow from the standard theorems for continuous dependence of solutions of ODEs on parameters, 
		because the number of ODEs in~\eqref{eqs:master-homog-complete} increases with~$M$, and becomes infinite in the limit, and also because of the 
		presence of the unbounded factor~$n$ on the right-hand sides  of~\eqref{eqs:master-homog-complete} and~\eqref{eq:master_inf-hom}. In Lemma~\ref{lem:Steve-convergence-homog-monotone} below, however,
		we will rigorously prove that
		\begin{equation}
			\label{eq:S^M->S-complete-hom-monotone}
			\lim_{M \to \infty} [S^n](t;M)  =  [S^n_{\infty}](t),  \qquad n \in \mathbb{N}.
		\end{equation}
		Therefore,  
		we can proceed to solve the infinite system~\eqref{eq:master_inf-hom}. 
		The ansatz 
		\begin{equation}
			\label{eq:inf_sol_hom}
			[S^n_{\infty}]=  [S_{\infty}]^n,\qquad n \in \mathbb{N}
		\end{equation}
		transforms the system~\eqref{eq:master_inf-hom} into
		\begin{equation*}
			n[S_{\infty}]^{n-1}\frac{d[S_{\infty}]}{dt} = -n(p+q)[S_{\infty}]^n+nq[S_{\infty}]^{n+1},
			\qquad [S_{\infty}](0)=1-I^0.
		\end{equation*}
		Dividing by $n[S_{\infty}]^{n-1}$, we find that the infinite system reduces to the single ODE 
		\begin{equation}
			\label{eq:homogeneous_Bass_frac-[S]-infty}
			\frac{d}{dt}[S_{\infty}]=-\left(p+q\right)[S_{\infty}]+  q[S_{\infty}]^2,
			\qquad [S_{\infty}](0)=1-I^0.
		\end{equation}
		
		Let $f^{\rm compart.} = 1- [S_{\infty}]$.  Then~$f^{\rm compart.}$
		satisfies~\eqref{eq:ODE-for-f^complete_infty}.  
		The limit~\eqref{eq:f_Bass-Niu(t)} follows
		from~\eqref{eq:S^M->S-complete-hom-monotone} and~\eqref{eq:inf_sol_hom} with $n=1$.
	\end{proof}
	
	In the proof of Theorem~\ref{thm:Niu(t)} we used the following convergence result:
	\begin{lemma}
		\label{lem:Steve-convergence-homog-monotone}
		For any $n \in \mathbb{N}$ and $t \ge 0$, the solution~$[S^n](t;M)$  of  equations~\eqref{eqs:master-homog-complete}
		converges monotonically as $M \to \infty$ to the solution $[S^n_\infty](t)$ of~\eqref{eq:master_inf-hom}.
	\end{lemma}
	\begin{proof}
		Let $n \in \mathbb{N}$ and let $M \ge n$. 
		Taking the integral of ODE~\eqref{eq:master-homog-complete} for $[S^n]$ 
		from zero to $t$ and using the initial condition~\eqref{eq:master-homog-complete-IC} 
		gives 
		\begin{equation}
			\label{eq:[S^n](t;M)-integral form}
			[S^n](t;M)-1 = -n
			\int_0^t \left(p(s)+q(s)\frac{M-n}{M-1}\right)[S^n](s;M) \,ds +n\frac{M-n}{M-1}\int_0^t q(s)[S^{n+1}](s;M) \,ds.
		\end{equation}
		Let us consider the limit of~\eqref{eq:[S^n](t;M)-integral form} as $M \to \infty$. 
		Since $[S^n](t;M)$ is monotonically decreasing in~$M$, see~\eqref{eq:y_n=S^n](t;M)-[S^n](t;M+1)} and~\eqref{eq:[S^n]-monotone-in-M}, and 
		since  $[S^n] \ge  0$ as a probability,
		this implies that $[S^n](t;M)$ converges pointwise as $M \to \infty$ 
		to some limit~$[S^n_\infty](t)$.
		Therefore, as $M \to \infty$,
		the left-hand side of~\eqref{eq:[S^n](t;M)-integral form} converges to 
		$[S^n_\infty]-1$. 
		In addition, since $[S^n]$ is a probability,  $0 \le [S^n](t;M) \le 1$, and so by
		the dominated convergence theorem, the integrals of~$[S^n]$ and~$[S^{n+1}]$ on the 
		right-hand side of~\eqref{eq:[S^n](t;M)-integral form}
		converge to the integrals of the limits. 
		Since the coefficients also converge, the limit of~\eqref{eq:[S^n](t;M)-integral form} as $M \to \infty$ is
		\begin{equation}
			\label{eq:[S^n_infty](t)-integral form}
			[S^n_\infty](t)-1 = -n
			\int_0^t \left(p(s)+q(s)\right)[S^n_\infty](s) \,ds +n\int_0^t q(s)[S^{n+1}_\infty](s) \,ds.
		\end{equation}
		
		Since  $[S^n_\infty](t)$ is the pointwise limit of a sequence of
		measurable functions, it is also measurable. Therefore, it follows from~\eqref{eq:[S^n_infty](t)-integral form} that it is 
		continuous, hence differentiable. Differentiating~\eqref{eq:[S^n_infty](t)-integral form},
		we conclude that  $[S^n_\infty]$ satisfies the limit ODE~\eqref{eq:master_inf-hom}.
	\end{proof}

		\section{Bass and SI~models on circles}
		\label{sec:1D_1_sided}

		Consider now the Bass and SI~models on the circle,
		where each node can only be influenced by its left and right neighbors. 
		We can allow peer effects to be {\em anisotropic}, so that the 
		influence rates of the left and right neighbors
		are~$q^{\rm L}$  and~$q^{\rm R}$, respectively.
		Thus,
		\begin{subequations}
			\label{eqs:circular-network}
			\begin{equation}
				\label{eq:p_j_q_j_twosided_circle}
				I_j^0\equiv I^0, \quad p_j\equiv p(t),\quad
				q_{k,j}\equiv
				q^{\rm L}(t) \, \mathbbm{1}_{(j-k)\, {\rm mod} \,  M=1}+
				q^{\rm R}(t) \, \mathbbm{1}_{(j-k)\, {\rm mod} \,  M=-1},
				\qquad  k,j\in {\cal M}.
			\end{equation} 
			%
			Hence, the adoption rate of~$j$ is
			\begin{equation}
				\label{eq:Bass-model-homog-circle-lambda_j}
				\lambda^{\rm circle}_j(t) := p+ q^{\rm L}(t) X_{j-1}(t)+q^{\rm R}(t) X_{j+1}(t), 
			\end{equation}
			where $X_{0}:=X_M$ and $X_{M+1}:=X_1$. The model parameters satisfy 
			\begin{equation}
				0 \le I^0<1, \qquad 	q^{\rm L}(t), q^{\rm R}(t)\ge 0, \quad  	q^{\rm L}(t)+q^{\rm R}(t)>0, \quad  p(t) \ge 0, \quad t>0,
			\end{equation}
			and 
			\begin{equation}
				I^0>0  \qquad  \text{or} \qquad  p(t)>0, \quad t>0. 
			\end{equation}
		\end{subequations}
	Furthermore, we assume that $p(t)$, $q^{\rm L}(t)$, and $q^{\rm R}(t)$ are piecewise continuous.
		We denote the expected adoption/infection level in the Bass/{\rm SI} model~{\rm (\ref{eqs:Bass-SI-models-ME},\ref{eqs:circular-network})} on the circle by~$f^{\rm circle}$.
		Specifically, we denote the expected adoption level in the Bass model
		by $f^{\rm circle}_{\rm Bass}:=f^{\rm circle}(\cdot,I^0=0)$, 
		and the expected infection level in the SI~model
		by~$f^{\rm circle}_{\rm SI}:=f^{\rm circle}(\cdot,p=0)$.
		
		\subsection{Monotone convergence of~$f^{\rm circle}$}
		\label{sec:circle-infinite}
		Consider the Bass/SI model~{\rm (\ref{eqs:Bass-SI-models-ME},\ref{eqs:circular-network})} on a circle. In this case, we see similar results to the complete network.
		
		\begin{lemma}
			\label{lem:f_circle_increases_with_M}
			Let $t>0$. 
			Then $f^{\rm circle}(t;p,q,I^0,M)$ is   monotonically increasing in~$M$. 
		\end{lemma}
			Using the monotonicity in~$M$, we can prove the convergence 
		of $f^{\rm circle}$  as $M \to \infty$
		and compute its limit:  
		\begin{theorem}
			\label{thm:f^circle->f^1D}
			Consider the Bass/{\rm SI} model~{\rm (\ref{eqs:Bass-SI-models-ME},\ref{eqs:circular-network})} on the circle. 
			Then
			\begin{subequations}
				\label{eqs:f^1D}
				\begin{equation}
					\label{eq:f_1D_onesided-M->infinity}
					\lim_{M\to\infty}f^{\rm circle}(t;p(t),q(t),I^0,M) = f^{\rm 1D}(t;p(t),q(t),I^0),
				\end{equation}
				where $f^{\rm 1D}$, the expected level of adoption/infection for the Bass/SI model
                                  on the one-dimensional lattice with  nearest-neighbor interactions,
				is the solution of
				\begin{equation}
					\label{eq:f^1D_der}
					\frac{d f}{dt}= (1-f)\Big(p(t)+q(t)\big(1-(1-I^0)e^{- \int_0^t p(\tau)}\big)\Big), \qquad f(0)=I^0.
				\end{equation}
			\end{subequations}
		\end{theorem}
	From Lemma~\ref{lem:f_circle_increases_with_M} and Theorem~\ref{thm:f^circle->f^1D} we have
	\begin{corollary}
		\label{cor:monotone-convergence-circle-to-f_1D}
		The convergence of~$f^{\rm circle}$ to~$f^{\rm 1D}$
		is monotone in~$M$. 
	\end{corollary}
	\begin{proof}
		This follows from Lemma~\ref{lem:f_circle_increases_with_M} and Theorem~\ref{thm:f^circle->f^1D}.
	\end{proof}

	We can also use the monotonicity to obtain lower and upper bounds for~$f^{\rm circle}$: 
	\begin{corollary}
		\label{cor:f_circle<f_1D}
		Consider the Bass/{\rm SI} model~{\rm (\ref{eqs:Bass-SI-models-ME},\ref{eqs:circular-network})} on the circle. 
		Then  
		\begin{equation}
			\label{eq:f_circle<f_1D}
			1-(1-I^0)e^{-\int_0^t p}<	f^{\rm circle}(t;p(t),q(t),I^0,M)< f^{\rm 1D}(t;p(t),q(t),I^0), \qquad t>0, \quad M=2,3, \dots, 
		\end{equation}
		where $f^{\rm 1D}$ 
		is the solution of~\eqref{eq:f^1D_der}.
	\end{corollary}
	\begin{proof}
		This follows from Lemma~\ref{lem:f_circle_increases_with_M},  eq.~\eqref{eq:[S]M=1}, and Theorem~\ref{thm:f^circle->f^1D}.
	\end{proof}
\subsection{Time-independent parameters}
When $p$ and $q$ are independent of time, we obtain from
Theorem~\ref{thm:f^circle->f^1D}  the explicit limits
$$
\lim_{M \to \infty} f^{\rm circle}_{\rm Bass}(t;p,q,M) =
f^{\rm 1D}_{\rm Bass}(t;p,q), \qquad 
f^{\rm 1D}_{\rm Bass}  :=1-e^{-(p+q)t+ q\frac{1-e^{-pt}}{p}}, 
$$
and
$$
\lim_{M \to \infty} f^{\rm circle}_{\rm SI}(t;q,I^0,M) =
f^{\rm 1D}_{\rm SI}(t;q, I^0), \qquad 
f^{\rm 1D}_{\rm SI}: =1- (1-I^0) e^{-qI^0t}.
$$
The former limit was first derived in~\cite{OR-10} and rigorously justified in~\cite{DCDS-23}, the latter limit is new.
The monotone convergence of~$f^{\rm circle}$ to~$f^{\rm 1D}$ is illustrated in Figure~\ref{fig:f_circle_converge2_f_1D}. 

\begin{figure}[!h]
	\begin{center}
		\scalebox{0.6}{\includegraphics{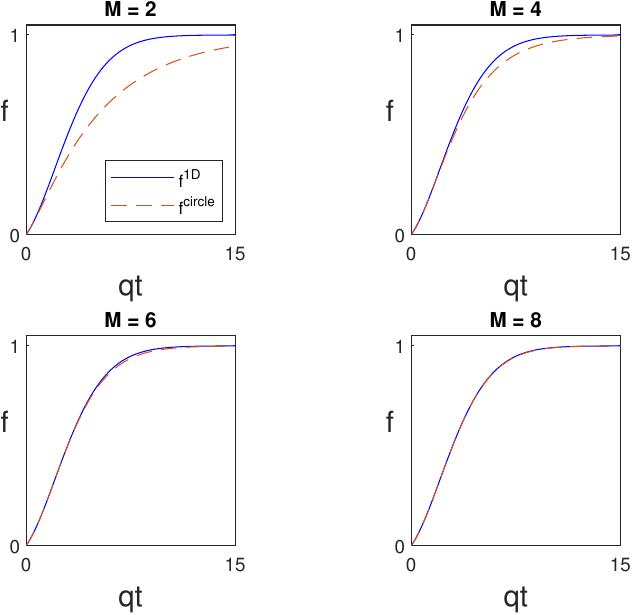}}
		\caption{Monotone convergence of~$f^{\rm circle}_{\rm Bass}$ (dashes) to~$f^{\rm 1D}_{\rm Bass}$ (solid).  Here~$\frac{q}{p} = 10$, $I^0=  0$, and $M=2,4,6,8$.}
		\label{fig:f_circle_converge2_f_1D}
	\end{center}
\end{figure}
		
			\subsection{Proof of Lemma~\ref{lem:f_circle_increases_with_M} and Theorem~\ref{thm:f^circle->f^1D}} 
		
		Let
		\[
		S^{n} :=S_{j+1,\ldots, j+n },
		\qquad n  \in {\cal M},
		\] 
		denote the event that the~$k$
		adjacent nodes $\{j+1,\ldots, j+n \}$ are nonadopters at time~$t$, 
		and let~$[S^{n}]$ denote the probability of this event.
		Note that the probabilities $\{[S^{n}]\}$ are independent of~$j$, because of translation invariance. 
		Then we have 
		\begin{lemma}[\cite{OR-10}]
			\label{lem_S^k_both}
			The reduced master equations for 
			the Bass/{\rm SI} model~{\rm (\ref{eqs:Bass-SI-models-ME},\ref{eqs:circular-network})} on the circle are
			\begin{subequations}
				\label{eqs:master-two-sided-circle}
				\begin{align}
					\label{eq:1D_both_Sk}
					&\frac{d [S^n]}{dt}{}= - \big(np(t)+q(t)\big)[S^n]+q(t)[S^{n+1}], \qquad n=1, \ldots, M-1, 
					\\  \label{eq:1D_both_SM}
					&\frac{d [S^M]}{dt}{}=-Mp(t)[S^M],
				\end{align}
				where
				$q(t)=q^{\rm R}(t)+q^{\rm L}(t)$,
				subject to the initial conditions
				\begin{equation}
					\label{eq:1D_both_Sk_ic}
					[S^n](0)= (1-I^0)^n, \qquad\qquad n=1, \ldots, M.
				\end{equation}
			\end{subequations}
		\end{lemma}

		\subsubsection{Monotonicity of $\{[S^n]\}$  in~$M$}

		\begin{proof}[Proof of Lemma~\ref{lem:f_circle_increases_with_M}]
			Let 
			$$
			y_n(t):= [S^n](t;M)-[S^n](t;M+1), 
			\qquad n\in {\cal M}.
			$$ 
			By the master equations~\eqref{eqs:master-two-sided-circle},
			$$
			\frac{d y_M}{dt} +Mp y_M =  q z_M(t), \qquad 
			y_M(0) = 0,
			$$
			where 
			$$
			z_M := [S^M](t;M+1)-[S^{M+1}](t;M+1)=  [IS^{M}](t;M+1),
			$$
			where  $[IS^{M}](t;M+1)$ is the probability that the nodes $\{1, \dots, M\}$ are nonadopters and the remaining node is an adopter.  
			Since $z_M(t)>0$ for $t>0$, 
			we have from Lemma~\ref{lem:dy_dt+cy>0} that 
			\begin{subequations}
				\label{eqs:monotone-S-circle}
				\begin{equation}
					y_M(t)>0, \qquad t>0.
				\end{equation}
				
				Similarly, by~\eqref{eqs:master-two-sided-circle},
				$$
				\frac{d y_n}{dt} +(np+q) y_n =  q y_{n+1}(t), \qquad 
				y_n(0) = 0, \qquad n = 1, \dots, M-1.
				$$
				Therefore, by Lemma~\ref{lem:dy_dt+cy>0},
				\begin{equation}
					y_{n+1}(t)>0, \quad t>0  \qquad \Rightarrow \qquad y_{n}(t)>0, \qquad t>0.
				\end{equation}
			\end{subequations}
			From relations~\eqref{eqs:monotone-S-circle}, 
			we get by reverse induction on~$n$ that
			\begin{equation}
				\label{eq:[S^n]-monotone-in-M-circle}
				y_n(t)>0, \qquad n \in {\cal M},
			\end{equation}
			i.e., that $\{[S^n](t;M)\}_{n \in \cal M}$ are monotonically decreasing in~$M$.
			In particular,
			$$
			0<y_1(t) = [S](t;M)- [S](t;M+1) = f^{\rm circle}(t;M+1)-f^{\rm circle}(t;M).
			$$
		\end{proof}
		

		The result and proof of Lemma~\ref{lem:f_circle_increases_with_M} 
		are similar to those in Lemma~\ref{lem:f_complete_monotone_in_M}
		for a complete network. Note, however, that while the addition of nodes  
		is accompanied by a reduction of the weight of the edges in a complete network,
		this is not the case on the circle.

		We can motivate the result of Lemma~\ref{lem:f_circle_increases_with_M}
		as follows. Any node $j \in \cal M$  adopts either externally or internally. In the latter case, 
		the adoption of~$j$ can be traced back to 
		an {\em adoption path} that starts from another node~$i$ that adopted externally, and progresses through a series of internal adoptions that ultimately reach node~$j$. The probability of~$j$ to adopt either externally or 
		internally due to some adoption path of length $\le M-1$, is the same
		on circles with~$M$ and with $M+1$~nodes.
		On the circle with $M+1$ nodes, however, $j$~can also adopt due to adoption paths of length~$M$.
		Therefore, its overall adoption probability on the larger circle is higher.
		
		\subsubsection{Convergence of~$f^{\rm circle}$}
		
		
		\begin{proof}[Proof of Theorem~\ref{thm:f^circle->f^1D}]
			The proof is similar to that for complete networks (Theorem~\ref{thm:Niu(t)}).
			Our starting point are the master equations~\eqref{eqs:master-two-sided-circle}.
			If we formally fix~$n$ and let $M\to\infty$ in~\eqref{eqs:master-two-sided-circle}, we get the limiting system
			\begin{equation}
				\label{eq:mastereq1dinf}
				\frac{d [S^n_{\infty}]}{dt}=-(np+q)[S^n_{\infty}]+q[S^{n+1}_{\infty}],\qquad [S^n_{\infty}](0)=(1-I^0)^n,\qquad n \in \mathbb{N}.
			\end{equation}
			This does not immediately imply that $\lim_{M \to \infty} [S^n] =  [S^n_{\infty}]$.
			Indeed, this limit does not follow from the standard theorems on continuous dependence of solutions of ODEs on parameters, 
			because the number of ODEs in~\eqref{eqs:master-two-sided-circle} increases with~$M$, and becomes infinite in the limit, and because of the 
			presence of the unbounded factor~$n$ on the right-hand sides of~\eqref{eqs:master-two-sided-circle} and~\eqref{eq:mastereq1dinf}. In Lemma~\ref{lem:circle-convergence-monotone} below, however,
			we will rigorously prove that for any $n \in \mathbb{N}$,
			\begin{equation}
				\label{eq:S^M->S-circle-hom-monotone}
				\lim_{M \to \infty} [S^n](t;M)  =  [S^n_{\infty}](t).
			\end{equation}
			Therefore,  
			we can proceed to solve the infinite system~\eqref{eq:mastereq1dinf}.	To do that, we note that the ansatz 
			\begin{equation}
				\label{eq:substitution-S^n-circle}
				[S^n_{\infty}]= \left((1-I^0)e^{-\int_0^t p(\tau)}\right)^{n-1}\, [S^{\rm 1D}]
			\end{equation}
			reduces the infinite system~\eqref{eq:mastereq1dinf} to the single ODE
			\begin{equation}
				\label{eq:S^1D_der}
				\frac{d [S^{\rm 1D}]}{dt}=-\Big(p+q(1-(1-I^0)e^{-\int_0^t p(\tau)})\Big)[S^{\rm 1D}], \qquad [S^{\rm 1D}](0)=1-I^0.
			\end{equation}
			Since $f^{\rm circle} = 1-[S^1]$ and $f^{\rm 1D} = 1-[S^{\rm 1D}]$,  
			the result follows from  relation~\eqref{eq:S^M->S-circle-hom-monotone} with $n=1$. 
			The monotonicity in~$M$ follows from Lemma~\ref{lem:f_circle_increases_with_M}.
			%
		\end{proof}
		
		The proof of Theorem~\ref{thm:f^circle->f^1D} makes use of 
		\begin{lemma}
			\label{lem:circle-convergence-monotone}
			For any $n \in \mathbb{N}$ and $t \ge 0$, the solution $[S^n](t;M)$   of the master equations~\eqref{eqs:master-two-sided-circle}
			converges monotonically as $M \to \infty$ to the solution $[S^n_\infty](t)$  of equation~\eqref{eq:mastereq1dinf}.
		\end{lemma}
		\begin{proof}
			The proof is nearly identical to that of Lemma~\ref{lem:Steve-convergence-homog-monotone}.
			Let $n \in \mathbb{N}$ and $M \ge n$. 
			Integrating the ODE~\eqref{eq:1D_both_Sk} for~$[S^n](t;M)$ 
			from zero to $t$ and using the initial condition~\eqref{eq:1D_both_Sk_ic} 
			gives 
			\begin{equation}
				\label{eq:[S^n](t;M)-integral form-circle}
				[S^n](t;M)-(1-I^0)^n = -
				\int_0^t \left(np(s)+q(s)\right)[S^n](s;M) \,ds +\int_0^t q(s)[S^{n+1}](s;M) \,ds.
			\end{equation}
			Let us consider the limit of~\eqref{eq:[S^n](t;M)-integral form-circle} as $M \to \infty$. 
			Since $[S^n](t;M)$ is monotonically decreasing in~$M$, see~\eqref{eq:[S^n]-monotone-in-M-circle}, and 
			since  $[S^n] \ge  0$ as a probability,
			this implies that $[S^n](t;M)$ converges pointwise as $M \to \infty$ 
			to some limit~$[S^n_\infty](t)$.
			Therefore, as $M \to \infty$,
			the left-hand side of~\eqref{eq:[S^n](t;M)-integral form-circle} converges to 
			$[S^n_\infty]-(1-I^0)^n$. 
			In addition, since $[S^n]$ is a probability, $0 \le [S^n](t;M) \le 1$, and so, 
			by the dominated convergence theorem, the integrals of~$[S^n]$ and~$[S^{n+1}]$ on the 
			right-hand side of~\eqref{eq:[S^n](t;M)-integral form-circle}
			converge to the integrals of the limits. 
			Hence, the limit of~\eqref{eq:[S^n](t;M)-integral form-circle} as $M \to \infty$ is
			\begin{equation}
				\label{eq:[S^n_infty](t)-integral form-circle}
				[S^n_\infty](t)-(1-I^0)^n = -
				\int_0^t \left(np(s)+q(s)\right)[S^n_\infty](s) \,ds + \int_0^t q(s)[S^{n+1}_\infty](s) \,ds.
			\end{equation}
			
			Since a pointwise limit of a sequence of 
			measurable functions is also measurable,  $[S^n_\infty](t)$ is  measurable. 
			Hence, it follows from~\eqref{eq:[S^n_infty](t)-integral form-circle} that it is 
			continuous, hence differentiable. Differentiating~\eqref{eq:[S^n_infty](t)-integral form-circle},
			we conclude that  $[S^n_\infty]$ satisfies the limit ODE~\eqref{eq:mastereq1dinf}.
		\end{proof}
	\section{Heterogeneous complete networks with two groups}
	\label{sec:groups}

Consider now the Bass/SI~model on a heterogeneous complete network with two groups, each of size $M$.
The parameters for node $k_n$ in group $k$ are\,\footnote{In this case, there is no normalization for which the maximal internal influence $\sum_{m \in \cal M}q_{m,k_n}(t)$ is independent of the network size~$M$.}  
\begin{equation*}
		I_{k_n}^0 \equiv I_k^0, \quad p_{k_n}(t)\equiv p_k(t),\quad q_{m,k_n}(t)\equiv \frac{q_k(t)}{2M},\qquad m\in {\cal M}, \quad n=1,\dots, M, \qquad k=1,2,
\end{equation*}
where ${\cal M}:=\{1, \dots, 2M\}$. 
Hence, the adoption rate of a node in group $k$ is
\begin{subequations}
		\label{eq:discrete_group_model}
		\begin{equation}
			\lambda_{k}(t)=p_{k}(t)+\frac{q_k(t)}{2M}\sum_{i=1}^{2M}X_i(t), \qquad k=1,2.
		\end{equation}
		The model parameters satisfy 
		\begin{equation}
			0 \le I_k^0<1, \quad 	q_k(t)>0, \quad  p_k(t) \ge 0, \quad t>0, \qquad k=1,2.
		\end{equation}
		and 
		\begin{equation}
		I_k^0+p_k(t)>0,  \qquad t>0, \qquad k=1,2.
		\end{equation}
\end{subequations}
Furthermore, we assume that $p_k(t)$ and $q_k(t)$ are piecewise continuous.
We denote the expected adoption/infection level in the Bass/{\rm SI} model~{\rm (\ref{eqs:Bass-SI-models-ME},\ref{eq:discrete_group_model})} on a complete network with two groups by~$f^{\rm 2-groups}$.
Specifically, we denote the expected adoption level in the Bass model
by $f^{\rm 2-groups}_{\rm Bass}:=f^{\rm 2-groups}(\cdot,I_1^0=I_2^0=0)$, 
and the expected infection level in the SI~model
by~$f^{\rm 2-groups}_{\rm SI}:=f^{\rm 2-groups}(\cdot,p_1=p_2=0)$.

\subsection{Monotone convergence of~$f^{\rm 2-groups}$}
We can show the monotone convergence of the Bass/SI model~{\rm (\ref{eqs:Bass-SI-models-ME},\ref{eq:discrete_group_model})} on a complete network with two groups:
	\begin{lemma}
	\label{lem:f_complete_monotone_in_M_2groups}
	Let $t>0$. 
	Then $f^{\rm 2-groups}(t;p_1(t),p_2(t),q_1(t),q_2(t),I_1^0,I_2^0,2M)$ is monotonically increasing in~$M$.
\end{lemma}

Using the monotonicity in $M$, we can prove the convergence of $f^{\rm 2-groups}$ as $M\to\infty$ and compute its limit:

	\begin{theorem}
	\label{thm:2groups_conv}
	Consider the Bass/SI model~{\rm (\ref{eqs:Bass-SI-models-ME},\ref{eq:discrete_group_model})} on a complete network with two groups. 
	Then
	\begin{subequations}
		\label{eqs:f^2groups}
		\begin{equation}
			\label{eq:f_2groups-M->infinity}
			\lim_{M\to\infty}f^{\rm 2-groups}(t;p_1(t),p_2(t),q_1(t),q_2(t),I_1^0,I_2^0,2M) = f^{\rm 2-groups}_{\infty}(t),
		\end{equation}
		where $f^{\rm 2-groups}_{\infty}:=f_1+f_2$, and 
		$f_1,f_2$ are the solutions of
		\begin{equation}
			\label{eq:master_inf_2groups}
			\begin{aligned}
				\frac{df_1}{dt}&=\left(\frac{1}{2}-f_1\right)\Bigg(p_1(t)+q_1(t)\left(f_1+f_2\right)\Bigg), \qquad f_1(0)=\frac{I_1^0}{2},\\
				\frac{df_2}{dt}&=\left(\frac{1}{2}-f_2\right)\Bigg(p_2(t)+q_2(t)\left(f_1+f_2\right)\Bigg), \qquad f_2(0)=\frac{I_2^0}{2}.
			\end{aligned}
		\end{equation}
	\end{subequations}
Here $0\leq f_i\leq \frac{1}{2}$ denotes the fraction of adopters from group $i$ in the population. 
\end{theorem}
From Lemma~\ref{lem:f_complete_monotone_in_M_2groups} and Theorem~\ref{thm:2groups_conv} we have

	\begin{corollary}
	\label{cor:monotone-convergence-2groups}
	The convergence of~$f^{\rm 2-groups}$ to~$f^{\rm 2-groups}_\infty$
	is monotone in~$M$. 
\end{corollary}
\begin{proof}
	This follows from Lemma~\ref{lem:f_complete_monotone_in_M_2groups} and Theorem~\ref{thm:2groups_conv}.
\end{proof}

The monotone convergence of $f^{\rm 2-groups}$ to $f^{\rm 2-groups}_\infty$ is illustrated in Figure~\ref{fig:f_groups_converge2_f_bass}.

\begin{figure}[!h]
	\begin{center}
		\scalebox{0.6}{\includegraphics{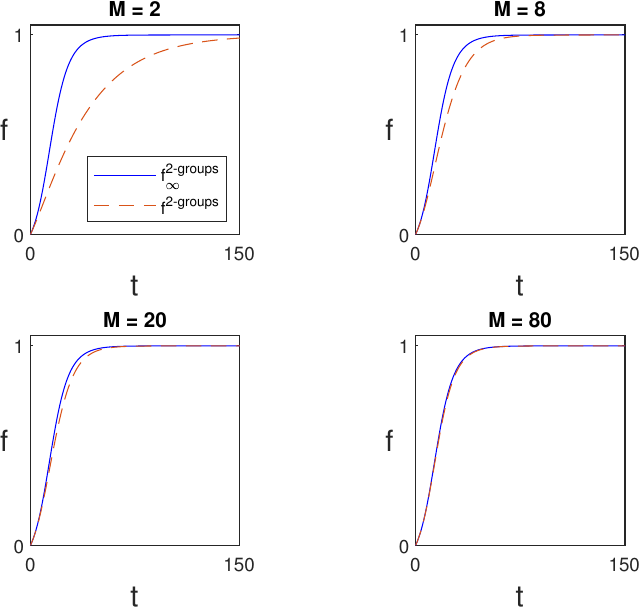}}
		\caption{Monotone convergence of~$f^{\rm 2-groups}$ (dashes) to~$f^{\rm 2-groups}_{\infty}$ (solid).  Here~$\frac{q_1}{p_1} =\frac{q_2}{p_2}=10$, $p_2=2p_1 $, $I^0=  0$, and $M=2,8,20,80$.}
		\label{fig:f_groups_converge2_f_bass}
	\end{center}
\end{figure}

\subsection{Proof of Lemma~\ref{lem:f_complete_monotone_in_M_2groups} and Theorem~\ref{thm:2groups_conv}}
	
Let 
$$
A_M:=   \{0,\dots,M\}^2 \setminus (0,0),
$$
	and let $[S^{k_1,k_2}](t;2M)$ denote the probability that $k_1$ nodes in group 1 and $k_2$ nodes in group 2 are non-adopters at time~$t$.
		Then
	$$
	f^{\rm 2-groups}= 
	f_1+f_2=1-\frac{1}{2}\left([S^{1,0}]+[S^{0,1}]\right).
	$$
	
			\begin{lemma}[\cite{DCDS-23}]
		\label{lem:master_2groups}
		The reduced master equations for 
		the Bass/SI model~{\rm (\ref{eqs:Bass-SI-models-ME},\ref{eq:discrete_group_model})} on a complete network with two groups are
		\begin{subequations}
			\label{eqs:master-2groups}
				\begin{equation}
				\label{eq:master_2groups}
				\begin{aligned}
					\frac{d[S^{k_1,k_2}]}{dt}&=-\left(k_1p_1(t)+k_2p_2(t)+\frac{2M-k_1-k_2}{2M}\left(k_1q_1(t)+k_2q_2(t)\right)\right)[S^{k_1,k_2}]\\&\quad+\left(\frac{M-k_1}{2M}[S^{k_1+1,k_2}]+\frac{M-k_2}{2M}[S^{k_1,k_2+1}]\right)\left(k_1q_1(t)+k_2q_2(t)\right), \qquad (k_1,k_2) \in A_M,
				\end{aligned}
			\end{equation}
			subject to the initial conditions
			\begin{equation}
				\label{eq:initial_2groups}
			[S^{k_1,k_2}](0)=(1-I_1^0)^{k_1}(1-I_2^0)^{k_2}, \qquad\qquad (k_1,k_2) \in A_M.
			\end{equation}
		\end{subequations}
	\end{lemma}
	
	\subsubsection{Monotonicity of $[S^{k_1,k_2}](t;2M)$}

	\begin{proof}[Proof of Lemma~\ref{lem:f_complete_monotone_in_M_2groups}]
		We proceed by reverse induction on $n = k_1+k_2$. 
		Let
		\begin{equation}
			\label{eq:y_n=S^n](t;M)-[S^n](t;M+1)_2groups}
			y_{k_1,k_2}(t):= [S^{k_1,k_2}](t;2M)-[S^{k_1,k_2}](t;2(M+1)),
			\qquad (k_1,k_2)\in A_M,
		\end{equation}
		where $[S^{k_1,k_2}](t;2M)$ is the solution of the master equations~\eqref{eqs:master-2groups}. 
		We begin with the induction base $n=2M$. 
		Then
		\begin{subequations}
			\label{eqs:d_dt_y_M-proof-monotonicity_2groups}
			\begin{equation}
				\frac{dy_{M,M}}{dt} +M\left(p_1(t)+p_2(t)\right) y_{M,M} =  z_{M,M}(t), \qquad 
				y_{M,M}(0) = 0,
			\end{equation}
			where 
			\begin{equation}
				\begin{aligned}
					z_{M,M} :=  \frac{M\left(q_1+q_2\right)}{2(M+1)}&\Bigg(\left([S^{M,M}](t;2(M+1))-[S^{M+1,M}](t;2(M+1))\right)\\&\quad+\left([S^{M,M}](t;2(M+1))-[S^{M,M+1}](t;2(M+1))\right)\Bigg).
				\end{aligned}	
			\end{equation}
		\end{subequations}
		Since 
		$$
		[S^{M,M}](t;2(M+1))-[S^{M+1,M}](t;2(M+1))= [I^{1,0}S^{M,M}](t;2(M+1))>0,$$ and $$
		[S^{M,M}](t;2(M+1))-[S^{M,M+1}](t;2(M+1))= [I^{0,1}S^{M,M}](t;2(M+1))>0,$$ 
		where $[I^{j_1,j_2}S^{k_1,k_2}]$ denotes the probability that there are $j_m$ adopters and $k_m$ nonadopters in group~$m$ for $m=1,2$, then $z_{M,M}(t)>0$ for $t>0$. Therefore, applying Lemma~\ref{lem:dy_dt+cy>0} to~\eqref{eqs:d_dt_y_M-proof-monotonicity_2groups} shows that
		\begin{equation}
			\label{eq:y_M(t)>0-complete_2groups}
			y_{M,M}(t)>0, \qquad t>0.
		\end{equation}
		
		Let $ n\in \{1,\dots,2M-1\}$. Then
		$$
		\begin{aligned}
			\frac{d[S^{k_1,k_2}]}{dt}(t;2M)=&-\left(k_1p_1+k_2p_2+\frac{2M-k_1-k_2}{2M}\left(k_1q_1+k_2q_2\right)\right)[S^{k_1,k_2}](t;2M)\\& +\left(\frac{M-k_1}{2M}[S^{k_1+1,k_2}](t;2M)+\frac{M-k_2}{2M}[S^{k_1,k_2+1}](t;2M)\right)\left(k_1q_1+k_2q_2\right), 	\end{aligned}
		$$
		subject to
		$$
		[S^{k_1,k_2}](t;2M)(0)=(1-I_1^0)^{k_1}(1-I_2^0)^{k_2}.
		$$
		Similarly,
		$$
		\begin{aligned}
			&\frac{d[S^{k_1,k_2}]}{dt}(t;2(M+1))=\\&\qquad-\left(k_1p_1+k_2p_2+\frac{2M+2-k_1-k_2}{2M+2}\left(k_1q_1+k_2q_2\right)\right)[S^{k_1,k_2}](t;2(M+1))\\ &\qquad+\left(\frac{M+1-k_1}{2M+2}[S^{k_1+1,k_2}](t;2(M+1))+\frac{M+1-k_2}{2M+2}[S^{k_1,k_2+1}](t;2(M+1))\right)\left(k_1q_1+k_2q_2\right),
		\end{aligned}
	$$
			subject to
		$$
	[S^{k_1,k_2}](t;2(M+1))(0)=(1-I_1^0)^{k_1}(1-I_2^0)^{k_2}.
		$$
		Taking the difference of these two equations gives
		\begin{subequations}
			\label{eqs:d_dt_y_n-proof-monotonicity_2groups}
			\begin{equation}
				\begin{aligned}
					\frac{d y_{k_1,k_2}}{dt} &+\left(k_1p_1+k_2p_2+\frac{2M+2-k_1-k_2}{2M+2}\left(k_1q_1+k_2q_2\right)\right) y_{k_1,k_2}\\ &= \left(k_1q_1+k_2q_2\right)\left(\frac{M+1-k_1}{2M+2}y_{k_1+1,k_2}+\frac{M+1-k_2}{2M+2}y_{k_1,k_2+1}\right)+ z_{k_1,k_2}(t), \quad 
					y_{k_1,k_2}(0) = 0,
				\end{aligned}
			\end{equation}
			where
			\begin{equation}
				\begin{aligned}
					z_{k_1,k_2} = \left(k_1q_1+k_2q_2\right)\Bigg(&\frac{-k_1}{M(2M+2)}\left(-[S^{k_1,k_2}](t;2M)+[S^{k_1+1,k_2}](t;2M)\right)\\&\qquad+\frac{-k_2}{M(2M+2)}\left(-[S^{k_1,k_2}](t;2M)+[S^{k_1,k_2+1}](t;2M)\right)\Bigg).
				\end{aligned}
			\end{equation}
		\end{subequations}
		Since
		$$
		\frac{-k_1}{M(2M+2)}\leq 0,\qquad \frac{-k_2}{M(2M+2)}\leq 0,
		$$
		and
		$$
		-[S^{k_1,k_2}]+[S^{k_1+1,k_2}]=  -[I^{1,0}S^{k_1,k_2}]<0, \qquad -[S^{k_1,k_2}]+[S^{k_1,k_2+1}]=  -[I^{0,1}S^{k_1,k_2}]<0 \qquad t>0,
		$$
		we have that $ z_{k_1,k_2}\ge 0$ for $t>0$. 
		
		Therefore, applying Lemma~\ref{lem:dy_dt+cy>0} to~\eqref{eqs:d_dt_y_n-proof-monotonicity_2groups} shows that for any $(k_1,k_2) \in A_M \setminus (M,M)$,
		\begin{equation}
			\label{eq:y_n+1(t)>0-complete_2groups}
			\Bigl\{y_{k_1+1,k_2}(t)\quad \text{and}\quad y_{k_1,k_2+1}(t)>0, \quad t>0\Bigr\}  \qquad \Rightarrow \qquad y_{k_1,k_2}(t)>0, \quad t>0.
		\end{equation}
		
		From relations~\eqref{eq:y_M(t)>0-complete_2groups} and~\eqref{eq:y_n+1(t)>0-complete_2groups} 
		we get by reverse induction on~$n$ that
		\begin{equation}
			\label{eq:[S^n]-monotone-in-M_2groups}
			y_{k_1,k_2}(t)>0, \qquad (k_1,k_2)\in A_M,
		\end{equation}
		i.e., that $\{[S^{k_1,k_2}](t;2M)\}$ are monotonically decreasing in~$M$.
	\end{proof}
	
	\subsubsection{Convergence of~$f^{\rm 2-groups}$}

\begin{proof}[Proof of Theorem~\ref{thm:2groups_conv}]
		Let $M\to\infty$. Then the master equations~\eqref{eq:master_2groups} converge to
		\begin{subequations}
			\label{eqs:inf_sys_groups}
	\begin{equation}
		\frac{d}{dt}[S^{k_1,k_2}_{\infty}]=-\left(k_1\left(p_1+q_1\right)+k_2\left(p_2+q_2\right)\right)[S^{k_1,k_2}_{\infty}]+\frac{1}{2}\left(k_1q_1+k_2q_2\right)\left([S^{k_1+1,k_2}_{\infty}]+[S^{k_1,k_2+1}_{\infty}]\right),
	\end{equation}
with the initial condition
\begin{equation}
	[S_\infty^{k_1,k_2}](0)=(1-I_1^0)^{k_1}(1-I_2^0)^{k_2}, \qquad (k_1,k_2)\in A_M.
\end{equation}
		\end{subequations}
	 We will prove in lemma~\ref{lem:2groups_convergence-monotone} below that $\lim_{M \to \infty} [S^{k_1,k_2}] =  [S^{k_1,k_2}_{\infty}]$.
	Substituting the ansatz
	$$
	[S^{k_1,k_2}_{\infty}]=[S^{1,0}_{\infty}]^{k_1}[S^{0,1}_{\infty}]^{k_2}.
	$$
	in~\eqref{eqs:inf_sys_groups} gives
	\begin{equation}
		\label{eq:ode_inf_sum}
		\begin{aligned}
k_1[S^{0,1}_{\infty}]\frac{d}{dt}[S^{1,0}_{\infty}]+k_2[S^{1,0}_{\infty}]\frac{d}{dt}[S^{0,1}_{\infty}]=&-\left(k_1\left(p_1+q_1\right)+k_2\left(p_2+q_2\right)\right)[S^{1,0}_{\infty}][S^{0,1}_{\infty}]\\&+\frac{1}{2}\left(k_1q_1+k_2q_2\right)\left([S^{1,0}_{\infty}]^2[S^{0,1}_{\infty}]+[S^{1,0}_{\infty}][S^{0,1}_{\infty}]^2\right).
		\end{aligned}		
	\end{equation}
	Let $[S^{1,0}_{\infty}]$ and $[S^{0,1}_{\infty}]$ be the solutions of
	\begin{equation}
		\label{eq:inf_system}
		\begin{aligned}
			\frac{d}{dt}[S^{1,0}_{\infty}]&=-\left(p_1+q_1\right)[S^{1,0}]+\frac{q_1}{2}\left([S^{1,0}_{\infty}]^2+[S^{1,0}_{\infty}][S^{0,1}_{\infty}]\right),\qquad [S_\infty^{1,0}](0)=1-I_1^0,\\
			\frac{d}{dt}[S^{0,1}_{\infty}]&=-\left(p_2+q_2\right)[S^{0,1}_{\infty}]+\frac{q_2}{2}\left([S^{0,1}_{\infty}]^2+[S^{1,0}_{\infty}][S^{0,1}_{\infty}]\right),\qquad [S_\infty^{0,1}](0)=1-I_2^0.
		\end{aligned}
	\end{equation}
Then $[S^{1,0}_{\infty}]$ and $[S^{0,1}_{\infty}]$ satisfy~\eqref{eq:ode_inf_sum}. Substituting $f_1:=\frac{1}{2}\left(1-[S^{1,0}_{\infty}]\right)$ and $f_2:=\frac{1}{2}\left(1-[S^{0,1}_{\infty}]\right)$ in~\eqref{eq:inf_system} gives
	\begin{equation}
		\begin{aligned}
			\frac{df_1}{dt}&=\left(\frac{1}{2}-f_1\right)\Bigg(p_1+q_1\left(f_1+f_2\right)\Bigg), \qquad f_1(0)=\frac{I_1^0}{2},\\
			\frac{df_2}{dt}&=\left(\frac{1}{2}-f_2\right)\Bigg(p_2+q_2\left(f_1+f_2\right)\Bigg), \qquad f_2(0)=\frac{I_2^0}{2}.
		\end{aligned}
	\end{equation}
	Since $f^{\rm 2-groups}_{\infty}:=f_1+f_2$, the result follows.
\end{proof}
	\begin{lemma}
		\label{lem:2groups_convergence-monotone}
		For any $t\ge 0, (k_1,k_2)\in A_M$, the solution $[S^{k_1,k_2}](t;2M)$ of the master equations~\eqref{eq:master_2groups}
		converges monotonically as $M \to \infty$ to the solution $[S^{k_1,k_2}_\infty](t)$  of equation~\eqref{eq:master_inf_2groups}
	\end{lemma}
	\begin{proof}
		The proof is nearly identical to that of Lemma~\ref{lem:Steve-convergence-homog-monotone}.
		Integrating the ODE~\eqref{eq:master_2groups} for~$[S^{k_1,k_2}](t;2M)$ 
		from zero to $t$ gives 
		{\small
		\begin{equation}
			\label{eq:[S^n](t;M)-integral form-2groups}
			\begin{aligned}
				&[S^{k_1,k_2}](t;2M)-(1-I_1^0)^{k_1}(1-I_2^0)^{k_2} =
				\\&\qquad  -\int_0^t\left(k_1p_1(s)+k_2p_2(s)+\frac{2M-k_1-k_2}{2M}\left(k_1q_1(s)+k_2q_2(s)\right)\right)
				[S^{k_1,k_2}](s;2M) \,ds \\&\qquad +\int_0^t\left(k_1q_1(s)+k_2q_2(s)\right)\left(\frac{M-k_1}{2M} [S^{k_1+1,k_2}](s;2M)\right) \,ds
				\\&\qquad +\int_0^t\left(k_1q_1(s)+k_2q_2(s)\right)\left(\frac{M-k_2}{2M} [S^{k_1,k_2+1}](s;2M)\right) \,ds.
			\end{aligned}
		\end{equation}
	}
		Let us consider the limit of~\eqref{eq:[S^n](t;M)-integral form-2groups} as $a \to \infty$. 
		Since $[S^{k_1,k_2}](t;2M)$ is monotonically decreasing in~$M$, see~\eqref{eq:[S^n]-monotone-in-M_2groups}, and 
		since  $[S^{k_1,k_2}] \ge  0$ as a probability,
		this implies that $[S^{k_1,k_2}](t;2M)$ converges pointwise as $M \to \infty$ 
		to some limit~$[S^{k_1,k_2}_\infty](t)$.
		Therefore, as $M \to \infty$,
		the left-hand side of~\eqref{eq:[S^n](t;M)-integral form-2groups} converges to 
		$[S^{k_1,k_2}_\infty]-1$. 
		In addition, since $[S^{k_1,k_2}]$ is a probability, $0 \le [S^{k_1,k_2}](t;2M) \le 1$, and so, 
		by the dominated convergence theorem, the integrals of~$[S^{k_1,k_2}], [S^{k_1+1,k_2}]$ and~$[S^{k_1,k_2+1}]$ on the 
		right-hand side of~\eqref{eq:[S^n](t;M)-integral form-2groups}
		converge to the integrals of the limits. 
		Hence, the limit of~\eqref{eq:[S^n](t;M)-integral form-2groups} as $M \to \infty$ is
		\begin{equation}
			\label{eq:[S^n_infty](t)-integral form-2groups}
			\begin{aligned}
				&[S^{k_1,k_2}_{\infty}](t;M)-(1-I_1^0)^{k_1}(1-I_2^0)^{k_2} =\\ &\qquad -\int_0^t\left(k_1p_1(s)+k_2p_2(s)+\left(k_1q_1(s)+k_2q_2(s)\right)\right)
				[S^{k_1,k_2}_\infty](s) \,ds \\&\qquad +\frac{1}{2}\int_0^t\left(k_1q_1(s)+k_2q_2(s)\right) [S^{k_1+1,k_2}_{\infty}](s) \,ds+\frac{1}{2}\int_0^t\left(k_1q_1(s)+k_2q_2(s)\right) [S^{k_1,k_2+1}_{\infty}](s) \,ds.
			\end{aligned}
		\end{equation}
		
		Since a pointwise limit of a sequence of 
		measurable functions is also measurable,  $[S^{k_1,k_2}_\infty](t)$ is  measurable. 
		Hence, it follows from~\eqref{eq:[S^n_infty](t)-integral form-2groups} that it is 
		continuous, hence differentiable. Differentiating~\eqref{eq:[S^n_infty](t)-integral form-2groups},
		we conclude that  $[S^{k_1,k_2}_\infty]$ satisfies the limit ODE~\eqref{eq:master_inf_2groups}.
	\end{proof}

		\bibliographystyle{plain}
		
	\end{document}